\documentclass[12pt,showkeys,showpacs,aip,jmp]{revtex4-1}

\usepackage{amsmath}
\usepackage{amsfonts}
\usepackage{amssymb}

\newcounter{theorem}
\renewcommand{\thetheorem}{\arabic{section}.\arabic{theorem}}
\newenvironment{thm}[1]{\par
\begin{sloppypar}\refstepcounter{theorem}%
\noindent{\bf #1 \thetheorem.}\it{}}{\end{sloppypar}}

\newenvironment{theorem}{\begin{thm}{Theorem}}{\end{thm}}
\newenvironment{proposition}{\begin{thm}{Proposition}}{\end{thm}}
\newenvironment{corollary}{\begin{thm}{Corollary}}{\end{thm}}
\newenvironment{lemma}{\begin{thm}{Lemma}}{\end{thm}}

\newenvironment{defi}[1]{\par
\begin{sloppypar}\refstepcounter{theorem}%
\noindent{\bf #1 \thetheorem.}\rm{}}{\end{sloppypar}}
\newenvironment{definition}{\begin{defi}{Definition}}{\end{defi}}

\newenvironment{remark}{\begin{defi}{Remark}}{\end{defi}}
\newenvironment{hypothesis}{\begin{defi}{Hypothesis}}{\end{defi}}

\newcommand{\eh}{\hfill}\newlength{\sperr}
\newenvironment{proof}{{\settowidth{\sperr}{\rm Proof}
\par\addvspace{0.3cm}\noindent\parbox[t]{1.3\sperr}{\rm P\eh r\eh o\eh o\eh
f\eh.}}}{\nopagebreak \mbox{}\hfill $\blacksquare $}

\def\l{\left}
\def\r{\right}

\def\R{{\rm I\kern-.2em R}}   
  \def\H{{\cal H}} 
\def\C{{\cal C}}   \def\Op{\mathfrak{Op}}
\def\ep{\epsilon}

\def\X{{\cal X}}
\def\gB{\Omega^{\ep}}
\def\N{\mathbb{N}}

\begin{document}

\title{On the continuity of spectra for families of magnetic pseudodifferential operators}

\date{\today}

\author{Nassim Athmouni}
\affiliation{Facult\'{e} des Sciences de l'Universit\'{e} de Sfax, Sfax, Tunisie}
\author{Marius M\u antoiu}
\affiliation{Departamento de Matem\'aticas, Universidad de Chile, Las Palmeras 3425, Casilla 653,
Santiago, Chile, Email: mantoiu@imar.ro}
\author{Radu Purice}
\affiliation{Institute of Mathematics Simion Stoilow of the Romanian Academy, P.O.  Box
1-764, Bucharest, RO-70700, Romania, Email: purice@imar.ro}

\begin{abstract}
For families of magnetic pseudodifferential operators defined by symbols and magnetic fields
depending continuously on a real parameter $\epsilon$, we show that the corresponding family of spectra also varies continuously
with $\epsilon$.
\end{abstract}


\maketitle

\section{Introduction and main result}

It is known \cite{AS,Ne,If,BC},  that "the spectrum of a
Schr\"odinger operator with magnetic field $B$ is continuous in
$B$" under some assumptions on the regularity of the magnetic
field. Following some ideas in \cite{Be} and \cite{BIT}, we would
like to put this result in a more general (abstract) perspective.

We shall work on the phase space $\Xi:=\X\times\X^*\equiv\mathbb{R}^n\times\mathbb{R}^n$ and use systematically notations of the form $X=(x,\xi)$, $Y=(y,\eta)$, ... for its points. We shall consider classical Hamiltonians
$h:\Xi\rightarrow\mathbb{R}$ (not having a simple specific form), defined on the phase space, smooth magnetic fields $B$ (closed 2-forms with bounded derivatives of
any order) and quantum Hamiltonians $H^A\equiv\mathfrak{Op}^A(h)$
defined by a choice of a vector potential $A$ (with $B=dA$)
\cite{MP1,MP3,KO1}. Our aim is to study the continuity properties
of the spectrum $\sigma(H^A)$ as a subset of $\mathbb{R}$ when
both the symbol and the magnetic field $B$ depend on a parameter
$\epsilon$ belonging to some interval $I$.

The main obstacles are the general form of the symbols $h^\ep$ and
the fact that $H^{A^\ep}$ is defined using the vector potential
$A^\ep$ which can be rather bad behaved even for bounded and
smooth magnetic fields $B^\ep$. To overcome this, we work only
with the magnetic symbol of the operators $H^{A^\ep}$ and we
obtain affiliation \cite{ABG,GI1} of the classical Hamiltonians
$h^\ep$ to a certain (not locally trivial) continuous field
\cite{Di} of twisted crossed-product $C^*$-algebras \cite{PR1}, defined only in terms of the magnetic fields
$\{B^\epsilon\}_{\epsilon\in I}$ \cite{MPR1}. In this way, the
problem is reduced to the study of the continuity properties in
$\ep$ of the magnetic symbols $r^\ep$ defining resolvent families
of the operators $H^{A^\ep}$. Then the results in \cite{Ri} directly
imply the outer continuity of the spectrum (i.e. the 'stability
of the spectral gaps') and the strong continuity in the regular
representation (that we shall prove) implies the inner continuity
of the spectrum (i.e. the 'stability of spectral islands').

To describe this result, we start recalling our version of covariant quantization in a magnetic field.
Given a continuous magnetic field $B=dA$ defined by a vector potential $A$, we have the following quantization rule
\cite{MP1,IMP,KO1}:
\begin{equation}
\left[\mathfrak{Op}^A(f)u\right](x):=(2\pi)^{-n}\int_\X dy\int_{\X^\star}d\eta\;e^{i(x-y)\cdot\eta}\,
\lambda^A(x,y)\,f\left(\frac{x+y}{2},
\eta\right)u(y),\label{porc}
\end{equation}
with
\begin{equation}
\lambda^A(x,y):=e^{-i\int_{[x,y]}A}=e^{-i(y-x)\cdot\int_0^1ds\;A(x+s(y-x))}.
\end{equation}
This is first defined for $f$ belonging to the Schwartz space $\mathcal S(\Xi)$ but extends to a topological
isomorphism \cite{MP1}
\begin{equation}
 \mathfrak{Op}^A:\mathcal{S}^\prime(\Xi)\rightarrow\mathbb{B}\big(\mathcal{S}(\mathcal{X});\mathcal{S}^\prime(\mathcal{X})\big)
\end{equation}
(with the dual Fr\'{e}chet topology on $\mathcal{S}^\prime(\Xi)$ and the strong topology on
$\mathbb{B}\big(\mathcal{S}(\mathcal{X});\mathcal{S}^\prime(\mathcal{X})\big)$).
The main reason to use (\ref{porc}) is {\it gauge-covariance}: equivalent choices of vector potentials lead to unitarily
equivalent operators.

Our quantization induces
on the algebra of observables a composition law that only depends on the magnetic field $B$ \cite{MP1,IMP}, requiring
\begin{equation}
 \mathfrak{Op}^A(f)\,\mathfrak{Op}^A(g)=:\mathfrak{Op}^A(f\,\sharp^B g)
\end{equation}
for any $f,g\in\mathcal{S}(\Xi)$. Explicitly we have
\begin{equation}
 \left(f\,\sharp^B g\right)(X)=\pi^{-2n}\int_\Xi \int_\Xi dY dZ \;
e^{-2i\sigma(X-Y,X-Z)}\Omega^B(x,y,z)f(Y)g(Z),
\end{equation}
where $\sigma$ is the canonical symplectic form on $\Xi$ and
$\Omega^B:=\exp{\{-i\Gamma^B\}}$ with $\Gamma^B(x,y,z)$ defined as
the flux of the magnetic field through the triangle
$<x-y-z,x+y-z,x-y+z>$:
\begin{equation}
\Gamma^B(x,y,z):=4\sum_{j,k=1}^ny_jz_k\int_0^1ds\int_0^{1-s}dt\,B_{jk}\big(x-y-z+2sy+2tz\big).\label{Omega-parametric}
\end{equation}

Extending $\sharp^B$ by duality, we get {\it the magnetic Moyal algebra}
$$
\mathfrak{M}^B(\Xi):=\left\{f\in\mathcal{S}^\prime(\Xi)\,\mid\,\forall g\in\mathcal{S}(\Xi),\
f\,\sharp^B g\in\mathcal{S}(\Xi),\ g\,\sharp^B f\in\mathcal{S}(\Xi)\right\}
$$
and {\it the algebra of bounded observables}
$$
\mathfrak{A}^B(\Xi):=\left(\mathfrak{Op}^A\right)^{-1}\left[\mathbb{B}\left(L^2(\mathcal{X})\right)\right].
$$
This second one will be a $C^*$-algebra isomorphic to $\mathbb{B}\left(L^2(\mathcal{X})\right)$; it depends on the
magnetic field but not on the vector potential, by gauge covariance.

Our aim is to show how our intrinsic observable algebra approach to the study of quantum Hamiltonians in
non-homogenuous magnetic fields allows for a proof of the continuity of the spectra with respect to very general variations
of the symbol and of the magnetic field.

Let us state the assumptions. We need H\"ormander's classes of symbols
$$
S^m_{\rho}(\Xi):=\left\{f\in
C^\infty(\Xi)\mid\forall(a,\alpha)\in\mathbb{N}^n\times\mathbb{N}^n,\,\exists
C_{a\alpha}>0,\ \left|(\partial^a_x\partial^\alpha_\xi f)(x,\xi)\right|\leq
C_{a\alpha}\left<\xi\right>^{m-\rho\vert\alpha\vert}\right\},
$$
where $\left<\xi\right>:=(1+|\xi|^2)^{1/2}$. Our previous results \cite{MP1,IMP} show that
$S^m_\rho(\Xi)\subset\mathfrak{M}^B(\Xi)$ and $S^0_0(\Xi)\subset\mathfrak{A}^B(\Xi)$.

On the symbol spaces ($m\in\R,\,\rho=1,0$) we have Fr\'echet structures defined by the families of semi-norms
indexed by $N,M\in\mathbb{N}$
\begin{itemize}
 \item $\|f\|_{(\Xi,1,m,N,M)}:=\underset{|a|\leq N}{\max}\,\underset{|\alpha|\leq M}{\max}\,\underset{(x,\xi)}
 {\sup}\left\vert\left<\xi\right>^{-m+|\alpha|}\big(\partial^a_x\partial^\alpha_\xi f\big)(x,\xi)\right\vert,
 \quad\forall f\in S^m_1(\Xi)$,
\item $\|f\|_{(\Xi,m,N,M)}:=\underset{|a|\leq N}{\max}\,\underset{|\alpha|\leq
M}{\max}\,\underset{(x,\xi)}{\sup}\left\vert\left<\xi\right>^{-m}\big(\partial^a_x\partial^\alpha_\xi f\big)(x,\xi)
\right\vert$,\ \ \ $\forall f\in S^m_0(\Xi)$.
\end{itemize}

We also recall that $f\in S^m_1(\Xi)$ is called {\it elliptic} (and we write $f\in S^m_{1,\text{\sf ell}}(\Xi)$) if
$$
|f(x,\xi)|\ge C\left<\xi\right>^m\ \ \ \ \ {\rm for}\ |\xi|\ {\rm big\ enough}.
$$

\medskip
\begin{hypothesis}\label{hyp-h}
Consider a family of Hamiltonians $\{h^\epsilon\}_{\epsilon\in I}$ with $I\subset\mathbb{R}$ a compact interval,
 such that
\begin{itemize}
 \item $h^\epsilon\in S^m_{1,\text{\sf ell}}(\Xi)$ with $m>0$, for each $\epsilon\in I$,
\item the map $I\ni\epsilon\mapsto h^\epsilon\in S^m_{1}(\Xi)$ is continuous for the Fr\'echet topology on $S^m_{1}(\Xi)$.
\item there exist $C\in\R$ such that $h^\ep\ge -C,\ \forall\ \ep\in I$.
\end{itemize}
\end{hypothesis}

\begin{hypothesis}\label{hyp-B}
We are given a family of magnetic fields $\{B^\epsilon\}_{\epsilon\in I}$ with the components
$B^\epsilon_{jk}\in BC^\infty(\mathcal{X})$ such that the map
$I\ni\epsilon\mapsto B^\epsilon_{jk}\in BC^\infty(\mathcal{X})$ is continuous for the Fr\'echet topology on
$BC^\infty(\mathcal{X})$.
\end{hypothesis}

\medskip
It has been shown in \cite{IMP} that real elliptic elements $f$ of $S^m_{1}(\Xi)$ define self-adjoint operators $\Op^A(f)$
in the Hilbert space
$\H:=L^2(\X)$, having as domain a suitable magnetic analog of the $m$'th order Sobolev space.
The semi-norms on $BC^\infty(\X)$ can be obtained from the expressions above for $\|\cdot\|_{(\Xi,m,N,M)}$, by replacing
$\Xi$ with $\X$ and by setting $m=0$.

In order to state our main result we recall some notions of continuity of subsets \cite{Be,BIT}.

\medskip
\begin{definition}
 Let $I$ be a compact interval and suppose given a family $\{\sigma^\ep\}_{\ep\in I}$ of closed subsets of
 $\mathbb{R}$.
\begin{enumerate}
\item The family $\{\sigma^\ep\}_{\ep\in I}$ is called {\it
outer continuous at} $\ep_0\in I$ if for any compact subset $K$ of $\mathbb R$ such that
$K\cap\sigma^{\ep_0}=\emptyset$, there exists a neighborhood $V^{\epsilon_0}_K$ of $\ep_0$ with
$K\cap\sigma^{\ep}=\emptyset$, $\forall \ep\in V^{\epsilon_0}_K$.
\item The family $\{\sigma^\ep\}_{\ep\in
I}$ is called {\it inner continuous at} $\epsilon_0\in I$ if  for any open subset $\mathcal O$ of $\mathbb R$ such
that $\mathcal O\cap\sigma^{\ep_0}\ne\emptyset$, there exists a neighborhood $V^{\epsilon_0}_\mathcal O\subset I$ of
$\ep_0$ with $\mathcal O\cap\sigma^{\ep}\ne\emptyset$, $\forall \ep\in V^{\epsilon_0}_\mathcal O$.
\end{enumerate}
\end{definition}

In \cite{Be} the sets $\sigma^\ep$ are compact and $K$ is only taken to be closed.

\medskip
\begin{theorem}\label{T-M}
 Suppose given a compact interval $I\subset\mathbb{R}$, a family of classical Hamiltonians $\{h^\ep\}_{\ep\in I}$
 satisfying Hypothesis \ref{hyp-h} and a family of magnetic fields $\{B^\ep\}_{\ep\in I}$ satisfying Hypothesis \ref{hyp-B}.
 Let us consider the family of quantum Hamiltonians $H^\ep:=\mathfrak{Op}^{A^\ep}(h^\ep)$ for some choice of a vector
 potential $A^\ep$ for $B^\ep$. Then the spectra $\sigma^\ep:=\sigma(H^\ep)\subset\mathbb{R}$ form an outer and inner
 continuous family at any point $\ep\in I$.
\end{theorem}

\medskip
Of course, if one only asks continuity conditions on the families $\{B^\ep\}_{\ep\in I}$ and $\{h^\ep\}_{\ep\in I}$ at
some point $\ep_0\in I$, the (outer and inner) continuity of the family of spectra will only be guaranteed at $\ep_0$.

Let us briefly comment upon the significance of Theorem \ref{T-M}:
\begin{itemize}
\item It extends the results in \cite{El,Be} to the case of
continuous models (with configuration space
$\mathcal{X}=\mathbb{R}^n$) and non-constant magnetic fields. We
mention in this context that our objects are no longer elements of
a crossed product but only unbounded observables affiliated to
twisted crossed-products with twisting cocycle in a rather complicated (not locally compact) group.
\item It extends the known results
\cite{Ne,If} to the class of elliptic symbols of any form and of
any strictly positive order. Notice that for Schr\"{o}dinger type
operators ($h^\ep(x,\xi)=\xi^2+V^\ep(x)$) the condition that the
components of the magnetic field should be smooth may be very much
weakened as we are going to show in future publication. \item It
is the first step in the study of the regularity of the spectral
bands and gaps with respect to variation of the magnetic field
(see \cite{Be}).
\end{itemize}

Our paper is devoted to the proof of Theorem \ref{T-M} and has the following structure.
In the next Section we present an abstract argument (following ideas and arguments in \cite{Be,El})
relating the statement of Theorem \ref{T-M} to the continuity of the symbols of the resolvents
of the family $\{H^\ep\}_{\ep\in I}$ in some special family of $C^*$-algebras, reducing the proof of Theorem \ref{T-M} to that
of Theorem \ref{T-M-1}. The third Section is devoted to
our main technical result proving the affiliation of the family $\{H^\ep\}_{\ep\in I}$
to a specific twisted crossed product $C^*$-algebra. In the 4-th Section we use the results in \cite{Ri} to prove that
this last twisted crossed product $C^*$-algebra is in fact an algebra of continuous sections in a field of $C^*$-algebras
and this is shown to be equivalent to our Theorem \ref{T-M-1}, thus finishing the proof of Theorem \ref{T-M}.

\section{The abstract part of the proof}

The abstract step in proving Theorem \ref{T-M} is to show how the norm of the resolvent
$R^\ep(\mathfrak{z}):=(H^\ep-\mathfrak{z}1)^{-1}$ is relevant for spectral continuity. In fact we have the following result.

\begin{proposition}\label{bazaka}
Suppose that $\{H^\ep\}_{\ep\in I}$ is a family of self-adjoint operators in the Hilbert space $\H$ such that for any
$\mathfrak{z}\notin\mathbb R$ the map
$$
I\ni\ep\mapsto\left\|\left(H^\ep-\mathfrak{z}1\right)^{-1}\right\|\in\mathbb{R}_+
$$
is upper (resp. lower) semi-continuous in $\ep_0\in I$. Then the spectra $\{\sigma(H^\ep)\}_{\ep\in I}$ form an outer
(resp. inner) continuous family of closed sets at the point $\ep_0\in I$.
\end{proposition}

\begin{corollary}
If  for any $\mathfrak{z}\notin\mathbb R$ the map
$$
I\ni\ep\mapsto\left\|\left(H^\ep-\mathfrak{z}1\right)^{-1}\right\|\in\mathbb{R}_+
$$
is continuous, then the family $\{\sigma(H^\ep)\}_{\ep\in I}$ is both outer and inner continuous.
\end{corollary}

\begin{proof}
For any $\mathfrak{z}\in\mathbb{C}\setminus\mathbb{R}$ the functions
$\kappa_\mathfrak{z}(t):=(t-\mathfrak{z})^{-1}$ belong to $C_0(\mathbb{R})$, i.e. they are continuous and small at infinity. Due to the Stone-Weierstrass Theorem and the
resolvent equation, their linear span is in fact an algebra and is dense in $C_0(\mathbb{R})$ for the norm
 $\|\cdot\|_\infty$. Thus, for any $\chi\in C_0(\mathbb{R})$ and for any $\delta>0$,
 there exist $N\in\mathbb{N}$, $a_j\in\mathbb{C}$ and $\mathfrak{z}_j\in\mathbb{C}\setminus\mathbb{R}$, with
 $j\in\{1,\ldots,N\}$, such that
\begin{equation}
 \|\chi-\sum_{j=1}^Na_j\,\kappa_{\mathfrak{z}_j}\|_\infty\leq\delta.
\end{equation}

By the functional calculus for self-adjoint operators we have
 $R^\ep(\mathfrak{z})=\kappa_{\mathfrak{z}}(H^\ep)$. We infer that for any $\chi\in C_0(\mathbb{R})$ the map $I\ni\ep\mapsto\|\chi(H^\ep)\|\in\mathbb{R}_+$ has the
same semi-continuity property as the map
$I\ni\ep\mapsto\|R^\ep(\mathfrak{z})\|\in\mathbb{R}_+$.

Let us suppose now upper semi-continuity in $\ep_0$ and assume that $\sigma\left(H^{\ep_0}\right)\cap K
=\emptyset$ for some compact set $K$. By Urysohn's Lemma,
there exists $\chi\in C_0(\mathbb R)_+$ with $\chi|_K=1$ and $\chi\vert_{\sigma\left(H^{\ep_0}\right)}=0$,
so $\chi\left(H^{\ep_0}\right)=0$. Choose a neighborhood $V$ of $\ep_0$ such that for $\ep\in V$
$$
\parallel \chi(H^{\ep})\parallel\,\le\,\parallel\chi(H^{\ep_0})\parallel+\frac{1}{2}=\frac{1}{2}.
$$
If for some $\ep\in V$ there exists $\lambda\in K\cap\sigma\left(H^{\ep}\right)$, then
$$
1=\chi(\lambda)\le\sup_{\mu\in\sigma^{\ep}}\chi(\mu)=\,\parallel \chi(H^{\ep})\parallel\,\le\frac{1}{2},
$$
which is absurd.

Let us assume now lower semi-continuity in $\ep_0$. Suppose
that there exist an open set $\mathcal{O}\subset\mathbb{R}$ such that
$\sigma\left(H^{\ep_0}\right)\cap\mathcal{O}\ne\emptyset$ and let $\lambda\in\sigma^\ep\cap\mathcal{O}$.
By Urysohn's Lemma there exist a positive function $\chi\in C_0(\mathbb{R})$ with
$\chi(\lambda)=1$ and $\text{supp}(\chi)\subset\mathcal{O}$; thus $\|\chi\left(H^{\ep_0}\right)\|\ge 1$.
Suppose moreover that for
any neighborhood $V\subset I$ of $\ep_0$ there exists $\ep\in V$ such that
$\sigma\left(H^{\ep}\right)\cap\mathcal{O}=\emptyset$ and thus $\chi\left(H^{\ep}\right)=0$. This clearly contradicts the lower
semi-continuity of $\ep\mapsto\|\chi\left(H^{\ep}\right)\|$. We conclude thus the inner continuity condition.
\end{proof}

Proving the continuity of the map
$I\ni\ep\mapsto\|R^\ep(\mathfrak{z})\|\in\mathbb{R}_+$ for any
$\mathfrak{z}\in\mathbb{C}\setminus\mathbb{R}$ is the aim of the
remaining part of the article. Our approach will be to work
intrinsically with the symbol
$r^\ep_{\mathfrak{z}}=\left(\mathfrak{Op}^{A^\ep}\right)^{-1}\!\!\left[R^\ep(\mathfrak{z})\right]$.
Clearly
$r^\ep_{\mathfrak{z}}\in\mathfrak{A}^\ep(\Xi)\equiv\mathfrak
A^{B^\ep}(\Xi)$ (it depends on the parameter $\ep\in I$ both
through the
 $\ep$-dependence of the symbol $h^\ep$ and through the $\ep$-dependence of the product $\sharp^{\ep}$, which is $B^\ep$
 -dependent) and
$\|r^\ep_{\mathfrak{z}}\|_\ep=\|R^\ep(\mathfrak{z})\|$ will be now
an $\ep$-dependent norm. However, no vector potential is in view
now. Our main technical result, Proposition \ref{prop-aff}, proven
in the next section allows us to control the inverse
$r^\ep_{\mathfrak{z}}$ of $h^\ep-\mathfrak{z}1$ in the Moyal
algebra $\mathfrak M^\ep(\Xi):= \mathfrak M^{B^\ep}(\Xi)$ with
respect to the product $\sharp^{\ep}\equiv\sharp^{B^\ep}$ for
$\ep$ fixed (a problem of affiliation), to prove that it belongs
in fact to some smaller algebra and to control the
$\ep$-dependence of the norms of the elements
$r^\ep_{\mathfrak{z}}\in\mathfrak C^\ep\subset
\mathfrak{A}^\ep(\Xi)$. These results will allow us to place
ourselves in the setting of continuous fields of $C^*$-algebras.
We shall prove the following statement

\begin{theorem}\label{T-M-1}
Suppose given a family of symbols $\{h^\ep\}_{\ep\in I}$
satisfying Hypothesis \ref{hyp-h} and a family of magnetic fields
$\{B^\ep\}_{\ep\in I}$ satisfying Hypothesis \ref{hyp-B}, then,
for any choice of vector potentials $\{A^\ep\}_{\ep\in I}$
associated to the magnetic fields $B^\ep$ ($B^\ep=dA^\ep$) and for
any $\mathfrak{z}\in\mathbb{C}\setminus\mathbb{R}$ the map
$$
I\ni\ep\mapsto\left\|\left(\mathfrak{Op}^{A^\ep}(h^\ep)-\mathfrak{z}1\right)^{-1}\right\|\in\mathbb{R}_+
$$
is continuous.
\end{theorem}

Thus, by the discussion above, we conclude that our Theorem \ref{T-M} is true.

\medskip
The main tool in proving Theorem \ref{T-M-1} will be to embed our
symbol algebras depending on $\ep\in I$ as "continuous sections in
the direct product" $\coprod_{\ep\in
I}\mathfrak{A}^\ep(\Xi)\rightarrow I$. We shall constantly use the
notation $BC_u(\mathcal{X})$ for the abelian $C^*$-algebra of all
bounded uniformly continuous complex functions on $\X$. We shall
construct these "continuous sections" by considering the twisted
crossed-products $BC_u(\mathcal{X})\rtimes^{\omega^\ep}_\theta\mathcal{X}$ for each
$\ep\in I$ (studied in \cite{MP1,MPR1}); here $\theta$
denotes the natural action of $\mathcal{X}$ on $BC_u(\mathcal{X})$
by translations and $\omega^\ep\equiv\omega^{B^\ep}$ is a group $2$-cocycle to be introduced below.

Let us consider the inverse partial Fourier transform
\begin{equation}
 \mathfrak{F}^-:\mathcal{S}(\Xi)\rightarrow\mathcal{S}(\mathcal{X}\times\mathcal{X}),\ \ \
 \left[\mathfrak{F}^-f\right](x,y):=\int_{\mathcal{X}^*}d\xi\,e^{i\xi\cdot y}f(x,\xi)
\end{equation}
(extended to $\mathcal{S}'(\Xi)$ and $L^2(\Xi)$). We can transport
the Moyal product $\sharp^B$ to a bilinear associative product on
$\mathcal{S}(\mathcal{X}\times\mathcal{X})$ and
$\mathfrak{F}^-\mathfrak{M}^B(\Xi)$ that we denote by
$\diamond^B$:
\begin{equation}
 \phi\diamond^B\psi:=\mathfrak{F}^-\left[(\mathfrak{F}\phi)\,\sharp^B(\mathfrak{F}\psi)\right].
\end{equation}
A simple computation gives
\begin{equation}\label{d-diam}
\left[\phi\diamond^B\psi\right](x,y)=\int_{\mathcal{X}}dz\,\phi(x+(z-y)/2,z)\,\psi(x+z/2,y-z)\,\omega^B(x-y/2;z,y-z),
\end{equation}
where
\begin{equation}\label{def-omega}
\omega^B(x;y,z):=\exp\{(-i\gamma^B(x,y,z))\}\in C\big(\mathcal{X};\mathbb{U}(1)\big)
\end{equation}
and $\gamma^B(x,y,z)$ is the flux of $B$ through the triangle
$<x,x+y,x+y+z>$. The group $C\big(\mathcal{X};U(1)\big)$
can be identified with the group of all the unitary elements in the $C^*$-algebra $BC_u(\X)$.

It is easy to see that the Banach space $\mathfrak
L:=L^1\big(\mathcal{X};BC_u(\mathcal{X})\big)$ is contained in
$\mathfrak{F}^-\left[\mathfrak{A}^B(\Xi)\right]$ and is also a
Banach $^*$-algebra under the multiplication $\diamond^B$ and the involution given by
$$
\phi^*(x;y)\equiv[\phi^*(y)](x):=\overline{\phi(x;-y)}.
$$

Let $\mathfrak{C}^B$ be the closure of $\mathfrak{F}\big(\mathfrak{L}\big)$ in $\mathfrak{A}^B\big(\Xi\big)$
(with the product $\sharp^B$). We shall also consider the
$C^*$-algebra $\mathfrak{B}^B:=\mathfrak{F}^-\mathfrak{C}^B$ (for
the product $\diamond^{B}$), that will be contained in
$\mathcal{S}^\prime(\mathcal{X}\times\mathcal{X})$. This
$C^*$-algebra is exactly the twisted crossed-product
$BC_u(\mathcal{X})\rtimes^{\omega^B}_\theta\mathcal{X}$ associated
to $BC_u(\mathcal{X})$, the action $\theta$ by translations of
$\mathcal{X}$ on $BC_u(\mathcal{X})$ and the 2-cocycle $\omega^B$
\cite{PR1,PR2,MPR1} and also the enveloping $C^*$-algebra of the
Banach $^*$-algebra $\mathfrak L$. Let us strengthen that the two
$C^*$-algebras $\mathfrak{B}^B$ and $\mathfrak{C}^B$ are
isomorphic. The constructions above can be performed for any of
the magnetic fields $B^\ep$, $\ep\in I$. We are going to use the
abbreviations $\mathfrak C^\ep:=\mathfrak C^{B^\ep}$, $\mathfrak
B^\ep:=\mathfrak B^{B^\ep}$.

In estimating some $C^*$-norms we shall need a special
representation, the left regular representation
\begin{equation}\label{reg-repr}
\Pi^\epsilon:\mathfrak{B}^\ep\rightarrow\mathbb{B}\big[L^2(\mathcal{X}\times\mathcal{X})\big],\
\quad\Pi^\epsilon(\phi)\,\psi:=\phi\diamond^\ep
\psi,\quad\forall\phi\in\mathfrak{B}^\ep,\quad\forall\psi\in\mathcal{H}:=L^2(\mathcal{X}\times\mathcal{X}).
\end{equation}
It really defines a representation, as one can easily notice using the results in \cite{MP1}.

We close this section by introducing a new algebra of
$\ep$-dependent symbols. We want to 'glue' all the 2-cocycles
$\omega^\ep(y,z)\in C(\mathcal{X};U(1))$ for $\ep\in I$ in a
single 2-cocycle with values in a larger group
$C(I\times\mathcal{X};U(1))$. This obliges us to also enlarge the
unital abelian algebra $BC_u(\mathcal{X})$ to the unital abelian
algebra $BC_u(I\times\mathcal{X})=C\big(I;BC_u(\mathcal{X})\big)$.
Let us consider the natural action $\widetilde{\theta}$ by
translations of $\mathcal{X}$ on $C\big(I;BC_u(\mathcal{X})\big)$,
given explicitely by
$\left[\widetilde{\theta}(x)f\right](\ep,z):=f(\ep,z+x)$, the
2-cocycle
$\widetilde{\omega}:\mathcal{X}\times\mathcal{X}\rightarrow
C(I\times\mathcal{X};U(1))$ given by
\begin{equation}\label{tilde-omega}
\left[\widetilde{\omega}(y,z)\right](\ep,x):=\left[\omega^\ep(y,z)\right](x)
\end{equation}
and the following composition law (similar to (\ref{d-diam})) on
$C_c\Big(\mathcal{X};C\big(I;BC_u(\mathcal{X})\big)\Big)$:
\begin{equation}\label{d-diam_tilde}
\left[\widetilde{\phi}\diamond\widetilde{\psi}\right](\ep,x,y):=\int_{\mathcal{X}}dz\,\widetilde{\phi}(\ep,x+(z-y)/2,z)
\,\widetilde{\psi}(\ep,x+z/2,y-z)\,\left[\widetilde{\omega}(z,y-z)\right](\ep,x-y/2).
\end{equation}
Taking the closure of
$C_c\Big(\mathcal{X};C\big(I;BC_u(\mathcal{X})\big)\Big)$ with respect to the norm
$$
\|\widetilde{\phi}\|_{\widetilde{\mathcal{L}}}:=\int_{\mathcal{X}}dx\
\|\widetilde{\phi}(x)\|_{C\big(I;BC_u(\mathcal{X})\big)}=
\int_{\mathcal{X}}dx\ \underset{\ep\in
I}{\sup}\left\|\left[\widetilde{\phi}(x)\right](\ep)\right\|_{BC_u(\mathcal{X})}=\int_{\mathcal{X}}dx\
\underset{\ep\in
I}{\sup}\underset{y\in\mathcal{X}}\,{\sup}\left|\left[\widetilde{\phi}(x)\right](\ep;y)\right|
$$
we obtain the space $\widetilde{\mathcal{L}}=
L^1\Big(\mathcal{X};C\big(I;BC_u(\mathcal{X})\big)\Big)$ that is a
Banach algebra for the composition (\ref{d-diam_tilde}). Let us consider its $C^*$-envelope that will be a crossed-product
$$
\mathfrak{B}:=C\big(I;BC_u(\mathcal{X})\big)\rtimes^{\widetilde{\omega}}_{\widetilde{\theta}}\mathcal{X}.
$$
We can then also define the isomorphic $C^*$-algebra
$\mathfrak{C}:=\mathfrak{F}\left[\mathfrak{B}\right]$.

An important remark is that for any $\ep\in I$ we have a natural
{\it evaluation} map
$$\mathfrak{e}_\ep:C_c\Big(\mathcal{X};C\big(I;BC_u(\mathcal{X})\big)\Big)\rightarrow
C_c\big(\mathcal{X};BC_u(\mathcal{X})\big),\quad\mathfrak{e}_\ep\big(\widetilde{\phi}\big)(x):=\left[\widetilde{\phi}(x)\right](\ep)\in
BC_u(\mathcal{X}),
$$
that extends by continuity to a contractive and surjective
projection (that we shall denote by the same symbol)
$\mathfrak{e}_\ep:\widetilde{\mathcal{L}}\rightarrow\mathcal{L}$
and to a contractive $C^*$-homomorphism
$\mathfrak{e}_\ep:\mathfrak{B}\rightarrow\mathfrak{B}^\ep$.

\section{An affiliation result}
As mentioned in the Introduction, in this Section we prove the affiliation of our family of Hamiltonians
$\{H^\ep\}_{\ep\in I}$ to a specific twisted crossed product $C^*$-algebra. Unfortunately,
no one of the affiliation results we have proved
already (see \cite{MPR2}, \cite{LMR}) implies directly Proposition
\ref{prop-aff}. Therefore we decided to give a full proof of the
statement. We shall mainly follow the arguments in \cite{MPR2},
keeping trace of the $\ep$-dependence and adding the necessary
technicalities in order to deal with $x$-dependent symbols.

 We first need notations for the norms defining the Fr\'echet topologies on various spaces.
 Let $C_{\text{\sf pol}}^\infty(\mathcal{X})$ be the space of indefinitely differentiable functions having polynomial growth as well as their derivatives. For $\varphi\in C_{\text{\sf pol}}^\infty(\mathcal{X})$, with $p\in\mathbb{R}$
 and $N\in\mathbb{N}$, we denote
 $$
 \|\varphi\|_{(\mathcal{X},p,N)}:=\underset{|a|\leq N}{\max}\,\underset{x\in\mathcal{X}}{\sup}\left\vert\left<x\right>^{-p}
 \big(\partial^a_x\varphi\big)(x)\right\vert;
 $$
thus the family
$\|\varphi\|_{(\mathcal{X},N)}:=\|\varphi\|_{(\mathcal{X},0,N)}$,
with $N\in\mathbb{N}$, defines the Fr\'echet topology on
$BC^\infty(\mathcal{X})$;
we also denote by
$\|\varphi\|_\infty:=\|\varphi\|_{(\mathcal{X},0)}$ the usual norm
on $BC(\mathcal{X})$. Associated to the above norms we can also consider
$$
\|\phi\|_{(\mathcal{X},p,N),(\mathcal{X},q,M)}:=\underset{|a|\leq N}{\max}\,\underset{|b|\leq M}{\max}
\,\underset{y\in\mathcal{X}}{\sup}\,\underset{z\in\mathcal{X}}{\sup}\,\left\vert\left<y\right>^{-p}\left<z\right>^{-q}
 \big(\partial^a_y\partial^b_z\phi\big)(y,z)\right\vert,
 $$
for all $\phi\in C_{\text{\sf
pol}}^\infty(\mathcal{X\times\mathcal{X}})$,
 for $p,q\in\mathbb R$ and $N,M\in\mathbb{N}$ and
$$
\|\widetilde\phi\|_{\infty,(\mathcal{X},p,N),(\mathcal{X},q,M)}:=\underset{x\in\mathcal{X}}
{\sup}\|\widetilde\phi(x)\|_{(\mathcal{X},p,N),
(\mathcal{X},q,M)}
$$
for any $\widetilde\phi\in
BC\Big(\mathcal{X};C_{\text{\sf
pol}}^\infty(\mathcal{X\times\mathcal{X}})\Big)$.
For $F\in S^{k_1,k_2}_{1,1}(\Xi\times\Xi)$ we shall need the
following family of norms
$$
\|F\|_{(k_1,N_1,M_1),(k_2,N_2,M_2)}:=
$$
$$
\underset{|a|\leq N_1}{\max}\,\underset{|\alpha|\leq
M_1}{\max}\,\underset{|b|\leq N_2}{\max}\,\underset{|\beta|\leq
M_2}{\max}\,\underset{Y\in\Xi}{\sup}\,\underset{Z\in\Xi}{\sup}\left|<\eta>^{-k_1+|\alpha|}<\zeta>^{-k_2+|\beta|}\left(\partial_y^a\partial_\eta^\alpha\partial_z^b
\partial_\zeta^\beta F^\ep\right)(Y,Z)\right|.
$$

Now let us introduce the main technical tool that will allow us to
estimate the oscillating integrals appearing in the computation of the symbol of the resolvent.

\begin{lemma}\label{M-est}
 We consider a function $\gamma\in C\Big(I;BC^\infty\big(\mathcal{X};C^\infty_{\hbox{\rm \tiny pol}}(\mathcal{X}\times\mathcal{X})\big)
 \Big)$ satisfying estimations of the form
\begin{equation}
C(\gamma)\equiv\underset{\ep\in
I}{\sup}\|\partial_x^\alpha\gamma^\ep\|_{\infty,(\mathcal{X},s_1(N_1,N_2),N_1),(\mathcal{X},s_2(N_1,N_2),N_2)}<\infty,\quad\forall\alpha\in\mathbb{N}^n
\end{equation}
and a function $F_\lambda\in
C\left(I;BC^\infty\big(\Xi;S^{k_1k_2}_{1,1}(\Xi\times\Xi)\big)\right)$
satisfying estimations of the form
\begin{equation}\label{l-est}
\underset{\ep\in
I}{\sup}\,\underset{x\in\mathcal{X}}{\sup}\left\Vert\left[\partial^b_x\partial^\alpha_\xi
F_\lambda^\ep\right](X)\right\Vert_{(k_1,N_1,M_1),(k_2,N_2,M_2)}\leq
C_{a,\beta}(F_\lambda)\left<\xi\right>^{-\rho-|\beta|}\lambda^{-\rho^\prime}
\end{equation}
for some strictly positive exponents $\rho$ and $\rho^\prime$.
Then the function
\begin{equation}\label{gF-est}
G(\ep;\lambda;X):=\int_{\Xi}\int_{\Xi}dYdZ\,e^{-2iz\cdot
\eta}\,e^{2iy\cdot\zeta}\,\gamma^\ep(x;y,z)\,F_\lambda^\ep(X;Y,Z)
\end{equation}
defines for each $\lambda>0$ an element of
$C\big(I;S^{-\rho}_1(\Xi)\big)$ and we have
$$
\underset{\ep\in
I}{\sup}\left\|G(\ep;\lambda;\cdot)\right\|_{\Xi,1,-\rho,N,M}\leq
C\lambda^{-\rho^\prime}.
$$
\end{lemma}

\begin{proof}
 We shall introduce some integrable factors into the integral (\ref{gF-est}) by applying suitable differential
 operators to the phase factor $e^{-2iz\cdot\eta}e^{2iy\cdot\zeta}$. In fact we have
$$
\left<y\right>^{-2}\left(1+\frac{1}{2i}y\cdot\partial_\zeta\right)e^{2iy\cdot\zeta}=e^{2iy\cdot\zeta},\
\ \ \ \
\left<z\right>^{-2}\left(1-\frac{1}{2i}z\cdot\partial_\eta\right)e^{-2iz\cdot\eta}=e^{-2iz\cdot\eta},
$$
$$
\left<\eta\right>^{-2}\left(1-\frac{1}{2i}\eta\cdot\partial_z\right)e^{-2iz\cdot\eta}=e^{-2iz\cdot\eta},\
\ \ \ \
\left<\zeta\right>^{-2}\left(1+\frac{1}{2i}\zeta\cdot\partial_y\right)e^{2iy\cdot\zeta}=e^{2iy\cdot\zeta}.
$$
We integrate by parts in (\ref{gF-est}), first with respect to the
$(y,z)$ variables (obtaining the integrable powers in $\eta$ and
$\zeta$ and some growing factors in $(y,z)$) and then with respect
to the $(\eta,\zeta)$ variables obtaining the integrable factors
in $y$ and $z$ due to the symbol behavior of the function
$F_\lambda$. More precisely, after $N_1+N_2+M_1+M_2$ integration
by parts we obtain the equality:
$$
\int_{\Xi}\int_{\Xi}dYdZ\,e^{-2iz\cdot
\eta}\,e^{2iy\cdot\zeta}\gamma^\ep(x;y,z)F_\lambda^\ep(X;Y,Z)=
$$
$$
=\int_{\Xi}\int_{\Xi}dYdZ\,e^{-2iz\cdot
\eta}\,e^{2iy\cdot\zeta}\times
\left[\left<y\right>^{-2M_2}\left(1-\frac{1}{2i}y\cdot\partial_\zeta\right)^{M_2}\left<z\right>^{-2M_1}
\left(1+\frac{1}{2i}z\cdot\partial_\eta\right)^{M_1}\right.\times
$$
$$
\times\left.\left<\eta\right>^{-2N_2}\left(1+\frac{1}{2i}\eta\cdot\partial_z\right)^{N_2}\left<\zeta\right>^{-2N_1}
\left(1-\frac{1}{2i}\zeta\cdot\partial_y\right)^{N_1}\left(\gamma^\ep
F_\lambda^\ep\right)\right](X;Y,Z)=
$$
$$
=\int_{\Xi}\int_{\Xi}dYdZ\,e^{-2iz\cdot
\eta}\,e^{2iy\cdot\zeta}\times
\left[\left<y\right>^{-M_2}\left<z\right>^{-M_1}\left(\frac{1}{
\left<y\right>}-\frac{y}{2i\left<y\right>}\cdot\partial_\zeta\right)^{M_2}\left(\frac{1}{\left<z\right>}+
\frac{z}{2i\left<z\right>}\cdot\partial_\eta\right)^{M_1}\right.\times
$$
$$
\times\left.\left<\eta\right>^{-N_2}\left<\zeta\right>^{-N_1}\left(\frac{1}{\left<\eta\right>}+\frac{\eta}{2i
\left<\eta\right>}\cdot\partial_z\right)^{N_2}
\left(\frac{1}{\left<\zeta\right>}-\frac{\zeta}{2i\left<\zeta\right>}\cdot\partial_y\right)^{N_1}\left(\gamma^\ep
F_\lambda^\ep\right)\right](X;Y,Z),
$$
where the differential polynomials have coefficients of class
$BC^\infty(\Xi)$. This clearly implies the estimation
$$
\left|\int_{\Xi}\int_{\Xi}dYdZ\,e^{-2iz\cdot
\eta}\,e^{2iy\cdot\zeta}\,\gamma^\ep(x;y,z)\,F_\lambda^\ep(X;Y,Z)
\right|\leq
$$
$$
\leq C_1(N_1,N_2,M_1,M_2)\,\underset{\ep\in
I}{\sup}\|\gamma^\ep\|_{(\mathcal{X},s_1(N_1,N_2),N_1),(\mathcal{X},s_2(N_1,N_2),N_2)}\,\underset{\ep\in
I}{\sup}\|F^\ep(X)\|_{(k_1,N_1,M_1),(k_2,N_2,M_2)}\times
$$
$$
\times
\int_{\Xi}dY\int_{\Xi}dZ\left<y\right>^{-M_2}\left<z\right>^{-M_1}\left<\eta\right>^{-N_2}
\left<\zeta\right>^{-N_1}\left<y\right>^{s_1(N_1,N_2)}
\left<z\right>^{s_2(N_1,N_2)}\left<\eta\right>^{k_1}\left<\zeta\right>^{k_2}\leq
$$
$$
\leq
C(N_1,N_2,M_1,M_2)\,C(\gamma)\,C(F_\lambda)\left<\xi\right>^{-\rho}\lambda^{-\rho^\prime},
$$
for a choice of the form $N_1>k_1+n$, $\,N_2>k_2+n$,
$\,M_1>s_1(N_1,N_2)+n$, $\,M_2>s_2(N_1,N_2)+n$.

In order to finish the proof we only have to apply a derivation
operator of the form $\partial^a_x\partial^\alpha_\xi$ to our
function $G(\ep;\lambda;(x,\xi))$ (defined in (\ref{gF-est})) and
remark that
\begin{itemize}
 \item for any multi-index $b\in\mathbb{N}^n$ the function $\left(\partial^b_x\gamma^\ep\right)(x;y,z)$
verifies exactly the same properties as $\gamma^\ep(x;y,z)$; \item
for any multi-indices $b\in\mathbb{N}^n$ and
$\beta\in\mathbb{N}^n$ the function
$\left(\partial^b_x\partial^\beta F\right)^\ep_\lambda(X;Y,Z)$
verifies exactly the same properties as $F^\ep_\lambda(X;Y,Z)$.
\end{itemize}
\end{proof}

Let us remark that under Hypothesis \ref{hyp-B} the 'magnetic
phase factor' $\gB$ in the explicit formula of the Moyal product
satisfies the hypothesis on .the functions $\gamma$ in the statement of the above Lemma \ref{M-est}

\begin{lemma}\label{B-est}
Assume Hypothesis \ref{hyp-B}. Then:
\begin{enumerate}
\item[{\rm (a)}] for each $\ep\in I$ and $x \in\mathcal{X}$,
$\;\!\gB(x;\cdot,\cdot) \in C^\infty_{\hbox{\rm \tiny
pol}}(\mathcal{X}\times\mathcal{X})$; \item[{\rm (b)}] for each
$N_1,N_2\in \N$, and $\alpha\in\mathbb{N}^n$ there exist
$s_1(N_1,N_2)\geq 0$ and $s_2(N_1,N_2)\geq 0$ such that
\begin{equation*}
\underset{\ep\in
I}{\sup}\|\partial_x^\alpha\Omega^\ep\|_{\infty,(\mathcal{X},s_1(N_1,N_2),N_1),(\mathcal{X},s_2(N_1,N_2),N_2)}<\infty;
\end{equation*}
\item[{\rm (c)}] the map $I\ni\ep\mapsto\gB\in
BC^\infty\big(\mathcal{X};C^\infty_{\hbox{\rm \tiny
pol}}(\mathcal{X}\times\mathcal{X})\big)$ is continuous.
\end{enumerate}
\end{lemma}

\begin{proof}
We use the explicit parametric form of $\gB$ in
(\ref{Omega-parametric}). Taking into account Hypothesis
\ref{hyp-B}, a simple examination of (\ref{Omega-parametric})
leads directly to the results. See also the proof of Lemma 4.2 in
\cite{IMP}.
\end{proof}

We come now to the main technical result of our paper.
\begin{proposition}\label{prop-aff}
 Under Hypothesis \ref{hyp-h} and \ref{hyp-B}, there exists some $a>0$ large enough such that for any
 $\mathfrak{z}\in\mathbb{C}\setminus[a,+\infty)$ we have:
\begin{enumerate}
\item for any $\ep\in I$, the function $h_\ep-\mathfrak{z}1\in
S^m_{1}(\Xi)\subset\mathfrak{M}^\ep(\Xi)$ is invertible for the
$\sharp^\ep$-product having an inverse $r^\ep_{\mathfrak{z}}\in
\mathfrak{F}\big[\mathfrak L\big]$; \item moreover the function
$I\times\Xi\ni(\ep,X)\mapsto
\widetilde{r}_{\mathfrak{z}}(\ep,X):=r^\ep_{\mathfrak{z}}(X)$
belongs to the algebra
$\mathfrak{F}\left[\widetilde{\mathcal{L}}\right]$ and
$\mathfrak{e}_\ep\big(\widetilde{r}_{\mathfrak{z}}\big)=r^\ep_{\mathfrak{z}}$.
\end{enumerate}
\end{proposition}

\begin{proof}
For $\lambda>0$ set $f^\ep:=h^\ep+\lambda$, consider the
point-wise inverse $(f^\ep)^{-1}=(h^\ep+\lambda)^{-1}\in
S^{-m}_1(\Xi)$ and compute (in the sense of distributions and
using oscillatory integral techniques relying on
$\exp\l[-2i\sigma(Y,Z)\r]$)
$$
\left[f^\ep\sharp^\ep \left(f^\ep\right)^{-1}\right](X)=
$$
$$
=\pi^{-2n}\int_\Xi\int_\Xi\,dY\,dZ\,e^{-2i\sigma(Y,Z)}\Omega^\ep(x;y,z)
\l\{1+\l[f^\ep(X-Y)-f^\ep(X-Z)\r]\left(f^\ep\right)^{-1}(X-Z)\r\}=
$$
$$
=1+\pi^{-2n}\int_\Xi\int_\Xi\,dY\,dZ\,e^{-2i\sigma(Y,Z)}\Omega^\ep(x;y,z)
\big[f^\ep(X-Y)-f^\ep(X-Z)\big]\left(f^\ep\right)^{-1}(X-Z)=
$$
$$
=1+\pi^{-2n}\int_\Xi\int_\Xi\,dY\,dZe^{-2iz\cdot
\eta}e^{2iy\cdot\zeta}\Omega^\ep(x;y,z)\times
$$
$$
\times\int_0^1ds\big[(z-y)\cdot(\partial_xf^\ep)(X-Z+s(Z-Y))+(\zeta-\eta)\cdot(\partial_\xi
f^\ep)(X-Z+s(Z-Y))\big]\left(f^\ep\right)^{-1}(X-Z).
$$
The fact that $\partial_\xi f^\ep\in S^{m-1}_1(\Xi)$ is important
in the following arguments. Although $\partial_x f^\ep$ only
belongs to $S^{m}_1(\Xi)$ one can write
$$
(z_j-y_j)\l[e^{-2iz\cdot
\eta}e^{2iy\cdot\zeta}\r]=-\frac{1}{2i}(\partial_{\eta_j}+\partial_{\zeta_j})\l[e^{-2iz\cdot
\eta}e^{2iy\cdot\zeta}\r]
$$
and after integrating by parts we get the same type of
improvement. Thus
\begin{equation}\label{sharp-ep}
\left[f^\ep\sharp^\ep
(f^\ep)^{-1}\right](X)=1+\sum_{k=1}^ng_j(\ep;\lambda;X),
\end{equation}
where
$$
g_j(\ep;\lambda;X):=\pi^{-2n}\int_\Xi\int_\Xi\,dY\,dZ\,e^{-2iz\cdot
\eta}\,e^{2iy\cdot\zeta}\,\Omega^\ep(x;y,z)\times
$$
$$
\times\l\{\frac{i}{2}\int_0^1ds\,(\partial_{\xi_j}\partial_{x_j}f^\ep)(X-Z+s(Z-Y))(f^\ep)^{-1}(X-Z)-\r.
$$
$$
-i\int_0^1ds\,(\partial_{x_j}f^\ep)(X-Z+s(Z-Y))\big(\partial_{\xi_j}
f^\ep\big)(X-Z)(f^\ep)^{-2}(X-Z)+
$$
$$
+\l.\int_0^1ds\,(\zeta_j-\eta_j)(\partial_{\xi_j}
f^\ep)(X-Z+s(Z-Y))(f^\ep)^{-1}(X-Z)\r\}=
$$
\begin{equation}\label{g-est}
=:\pi^{-2n}\int_\Xi\int_\Xi\,dY\,dZ\,e^{-2iz\cdot
\eta}\,e^{2iy\cdot\zeta}\,\Omega^\ep(x;y,z)\,F_\lambda^\ep(X;Y,Z).
\end{equation}

We intend to prove that each term $g_j(\ep;\lambda;\cdot)$ is a
symbol of strictly negative order with a uniform bound controlled
by $\lambda>0$. One has
$$
\left[\partial_x^a\partial_\xi^\alpha
g_j\right](\ep;\lambda;X)=\sum_{b\leq
a}C_a^b\,\pi^{-2n}\int_\Xi\int_\Xi\,dY\,dZ\,e^{-2iz\cdot
\eta}e^{2iy\cdot\zeta}\left[\partial^{a-b}_x\Omega^\ep\right](x;y,z)\left[\partial_x^b\partial_\xi^\alpha
F_\lambda^\ep\right](X;Y,Z).
$$
Due to our Hypothesis \ref{hyp-B}, all the derivatives
$\partial^{c}_x\Omega^\ep$ are dealt with by Lemma \ref{B-est}. We
now study the functions $\partial_x^b\partial_\xi^\alpha F^\ep$
with $F^\ep$ defined in (\ref{g-est}); any of these functions is
the sum of three contributions.
\begin{itemize}
 \item Let us begin with
$$
\partial_x^b\partial_\xi^\alpha\big[(\partial_{\xi_j}\partial_{x_j}f^\ep)(X-Z+s(Z-Y))(f^\ep)^{-1}(X-Z)\big]=
$$
\end{itemize}
$$
=\sum_{c\leq
b,\beta\leq\alpha}C_{\alpha,\beta}^{b,c}\left[\partial_x^{b-c}
\partial_\xi^{\alpha-\beta}\partial_{\xi_j}\partial_{x_j}
f^\ep\right](X-Z+s(Z-Y))\left[\partial_x^c\partial_\xi^\beta
(f^\ep)^{-1}\right](X-Z).
$$
We have to use the fact that, due to the ellipticity of $h$,
$|(f^\ep)^{-1}(X)|\leq
C(f)\left(\left<\xi\right>^m+\lambda\right)^{-1}$. Moreover, it is
straightforward to see by induction that for $|c|+|\beta|\geq1$
$$
\partial_x^c\partial_\xi^\beta (f^\ep)^{-1}=(f^\ep)^{-1}\mathfrak{s}^\ep_{c,\beta}(f),
\ \ {\rm with}\ \mathfrak{s}_{c,\beta}(f)\in
C\left(I;S^{-|\beta|}_1(\Xi)\right),
$$
verifying
$$
\left\|\mathfrak{s}_{c,\beta}(f)\right\|_{\Xi,-|\beta|,N,M}\leq
C\left\|f^\ep\right\|_{\Xi,m,N+|c|,M+|\beta|}^{|c|+|\beta|}.
$$

For $|b|+|\alpha|=0$ we have
$$
\left|\left[\partial_{\xi_j}\partial_{x_j}f^\ep\right](X-Z+s(Z-Y))\left[(f^\ep)^{-1}\right](X-Z)\right|\leq
$$
$$
\le
C(f)\left\|f^\ep\right\|_{\Xi,m,1,1}\frac{\left<\xi\right>^{m-1}\left<\zeta\right>^{m-1}\left<\eta\right>^{m-1}}
{\left<\xi-\zeta\right>^m+\lambda}\leq
C(f)\left\|f^\ep\right\|_{\Xi,m,1,1}\frac{\left<\xi\right>^{m-1}\left<\zeta\right>^{m-1}
\left<\eta\right>^{m-1}}{\left<\xi\right>^m/\left<\zeta\right>^m+\lambda}\leq
$$
$$
\leq
C(f)\left\|f^\ep\right\|_{\Xi,m,1,1}\frac{\left<\zeta\right>^{2m-1}\left<\eta\right>^{m-1}
\left<\xi\right>^{m-1}}{\left<\xi\right>^m+\lambda\left<\zeta\right>^m}\leq
C(f)\left\|f^\ep\right\|_{\Xi,m,1,1}\frac{\left<\zeta\right>^{2m-1}\left<\eta\right>^{m-1}
\left<\xi\right>^{m-1}}{\left<\xi\right>^m+\lambda}.
$$
We use now the inequality $a+b\geq(\mu a)^{\mu^{-1}}(\nu
b)^{\nu^{-1}}$, valid for $a,b\in\mathbb{R}_+,\
\,\mu,\nu\in[1,+\infty),\ \mu^{-1}+\nu^{-1}=1$. Thus for some
$\nu\geq 1$ such that $m\nu^{-1}-1<0$, we have
$$
\left|\left[\partial_{\xi_j}\partial_{x_j}f^\ep\right](X-Z+s(Z-Y))\left[(f^\ep)^{-1}\right](X-Z)\right|\leq
$$
$$
\le C(f)\left\|f^\ep\right\|_{\Xi,m,1,1}\left<\zeta\right>^{2m-1}\left<\eta\right>^{m-1}\frac{\left<\xi\right>^{m-1}}
{\left<\xi\right>^m+\lambda}\leq
$$
$$
\leq
C(f)\,C(\nu)\left\|f^\ep\right\|_{\Xi,m,1,1}\left<\zeta\right>^{2m-1}\left<\eta\right>^{m-1}
\left<\xi\right>^{m\nu^{-1}-1}\lambda^{-\nu^{-1}}.
$$
For any fixed $X\in\Xi$ we have, by very similar arguments, the
following estimations:
$$
\left|\partial_y^b\partial_\eta^\alpha\left[\left(\partial_{\xi_j}\partial_{x_j}f^\ep\right)(X-Z+s(Z-Y))
\left[(f^\ep)^{-1}\right](X-Z)\right]\right|\leq
$$
$$
\leq
C(f)\,C(\nu)\left\|f^\ep\right\|_{\Xi,m,1,1}\left<\zeta\right>^{2m-1-|\alpha|}\left<\eta\right>^{m-1-|\alpha|}
\left<\xi\right>^{m\nu^{-1}-1-|\alpha|}\lambda^{-\nu^{-1}}.
$$
Let us remark that
$$
\partial_z^b\partial_\zeta^\alpha\left[\left(\partial_{\xi_j}\partial_{x_j}f^\ep\right)(X-Z+s(Z-Y))
\left[(f^\ep)^{-1}\right](X-Z)\right]=
$$
$$
=(-1)^{|\alpha|}\sum_{c\leq
b,\beta\leq\alpha}C_{\alpha,\beta}^{b,c}\,(1-s)^{|\alpha|-|\beta|}\left[\partial_x^{b-c}\partial_\xi^{\alpha-\beta}
\partial_{\xi_j}\partial_{x_j}f^\ep\right](X-Z+s(Z-Y))\left[\partial_x^c\partial_\xi^\beta
(f^\ep)^{-1}\right](X-Z),
$$
and thus is completely similar to the higher order derivatives
$\partial_x^b\partial_\xi^\alpha$ that we shall now study.

For $|b|+|\alpha|>0$ we obtain
$$
\left|\partial_x^b\partial_\xi^\alpha\big[(\partial_{\xi_j}\partial_{x_j}f^\ep)(X-Z+s(Z-Y))(f^\ep)^{-1}(X-Z)\big]\right|\leq
$$
$$
\leq \left|f_\ep^{-1}(X-Z)\right|\sum_{c\leq
b}C_b^c\sum_{\beta\leq\alpha}C_\alpha^\beta\left\|f^\ep\right\|_{\Xi,m,|b-c|+1,|\alpha-\beta|+1}
\left[\left<\xi\right>\left<\eta\right>\left<\zeta\right>\right]^{m-|\alpha-\beta|-1}
\left|\mathfrak{s}_{c,\beta}(f)(X-Z)\right|\leq
$$
$$
\leq
C(b,\alpha)\left\|f^\ep\right\|_{\Xi,m,|b|+1,|\alpha|+1}^{|b|+1+|\alpha|+1}\frac{\left[\left<\xi\right>\left<\eta\right>
\left<\zeta\right>\right]^{m-|\alpha|-1}}{\left<\xi-\zeta\right>^m+\lambda}\leq
$$
$$
\leq
C(b,\alpha)C(\nu)\left\|f^\ep\right\|_{\Xi,m,|b|+1,|\alpha|+1}^{|b|+1+|\alpha|+1}\left<\eta\right>^{m-|\alpha|-1}
\left<\zeta\right>^{2m-|\alpha|-1}\left<\xi\right>^{m\nu^{-1}-1-|\alpha|}\lambda^{-\nu^{-1}}
$$
and similar estimations for this term after application of
differential operators of the form
$\partial^b_y\partial^\alpha_\eta$ or
$\partial^b_z\partial^\alpha_\zeta$. Thus, for each $\ep\in I$
this first term (that we denote by $F^\ep_{\lambda,1}$) satisfies:
\begin{equation}\label{F1}
 F^\ep_{\lambda,1}\in BC^\infty\big(\Xi;S^{m-1,2m-1}_{1,1}(\Xi\times\Xi)\big)
\end{equation}
and
\begin{equation}\label{F1-2}
\underset{\ep\in
I}{\sup}\underset{x\in\mathcal{X}}{\sup}\left\|\left[\partial^b_x\partial^\alpha_\xi
F^\ep_{\lambda,1}\right](X;\cdot,\cdot)\right\|_{(m-1,N_1,M_1),(2m-1,N_2,M_2)}\leq
\end{equation}
$$
\leq
C(b,\alpha)C(\nu)C(f)\|f^\ep\|_{\Xi,m,|b|+N_1+N_2+1,|\alpha|+M_1+M_2+1}^{2+|b|+|\alpha|+N_2+M_2}
\left<\xi\right>^{m\nu^{-1}-1-|\alpha|}\lambda^{-\nu^{-1}}.
$$

Let us study its continuity with respect to $\ep\in I$. First, we
fix some $\ep\in I$ and for any $\ep^\prime\in I$ we consider the
difference
$$
\left(\partial_{\xi_j}\partial_{x_j}f^{\ep^\prime}\right)(X-Z+s(Z-Y))(f^{\ep^\prime})^{-1}(X-Z)-
(\partial_{\xi_j}\partial_{x_j}f^\ep)(X-Z+s(Z-Y))f_\ep^{-1}(X-Z)=
$$
$$
=\left[(\partial_{\xi_j}\partial_{x_j})(f^{\ep^\prime}-f_\ep)(X-Z+s(Z-Y))\right](f^{\ep^\prime})^{-1}(X-Z)+
$$
$$
+(\partial_{\xi_j}\partial_{x_j}f^\ep)(X-Z+s(Z-Y))\left[(f^{\ep^\prime})^{-1}-(f^\ep)^{-1}\right](X-Z)=
$$
$$
=\left[(\partial_{\xi_j}\partial_{x_j})(f^{\ep^\prime}-f^\ep)(X-Z+s(Z-Y))\right](f^{\ep^\prime})^{-1}(X-Z)+
$$
$$
+(\partial_{\xi_j}\partial_{x_j}f^\ep)(X-Z+s(Z-Y))\left[(f^{\ep^\prime})^{-1}(f^{\ep^\prime}-f^\ep)
\left[(f^\ep)^{-1}\right]\right](X-Z).
$$
After applying the operator $\partial_x^a\partial_\xi^\alpha$ to
the above difference, for the first term we can directly use the
previous analysis with $f^\ep$ replaced by $f^{\ep^\prime}-f^\ep$
and obtain a uniform bound with some constant multiplied by
$$
\left\|f^{\ep^\prime}-f^\ep\right\|_{\Xi,m,|b|+1,|\alpha|+1}^{|b|+1+|\alpha|+1},
$$
that converges to zero for $\ep^\prime\rightarrow\ep$ due to our
Hypothesis \ref{hyp-h}. For the second term, we have to replace in
the previous analysis the factor $
(\partial_{\xi_j}\partial_{x_j}f^\ep)(X-Z+s(Z-Y)) $ that is an
element of $S^{m-1}_1(\Xi)$ and thus satisfies an estimation of
the type
$$
(\partial_x^b\partial_\xi^\alpha\partial_{\xi_j}\partial_{x_j}f^\ep)(X-Z+s(Z-Y))\leq
\|f^\ep\|_{\Xi,m,|b|+1,|\alpha|+1}\left<\xi\right>^{m-1-|\alpha|}\left<\eta\right>^{m-1-|\alpha|}
\left<\zeta\right>^{m-1-|\alpha|},
$$
with the derivatives $\partial_x^b\partial_\xi^\alpha$ of the
factor
$$
(\partial_{\xi_j}\partial_{x_j}f^\ep)(X-Z+s(Z-Y))(f^{\ep^\prime})^{-1}(X-Z)(f^{\ep^\prime}-f^\ep)(X-Z)
$$
that are bounded by
$$
C(b,\alpha)\|f^{\ep^\prime}-f^\ep\|_{\Xi,m,|b|+1,|\alpha|+1}\,\|f^\ep\|_{\Xi,m,|b|+1,|\alpha|+1}^{3+|b|+|\alpha|}
\left<\xi\right>^{m-1-|\alpha|}\left<\eta\right>^{m-1-|\alpha|}\left<\zeta\right>^{3m-1-|\alpha|}.
$$
Similar estimations are obtained analogously, applying
differential operators of the form
$\partial^b_y\partial^\alpha_\eta$ or
$\partial^b_z\partial^\alpha_\zeta$. Thus
\begin{equation}\label{F1-cont}
\underset{x\in\mathcal{X}}{\sup}\left\|\left[\partial^b_x\partial^\alpha_\xi
\left(F^\ep_{\lambda,1}-F^\ep_{\lambda,1}\right)\right](X;\cdot,\cdot)\right\|_{(m-1,N_1,M_1),(3m-1,N_2,M_2)}\leq
\end{equation}
$$
\leq
C(b,\alpha)\|f^\ep\|_{\Xi,m,|b|+N_1+N_2+1,|\alpha|+M_1+M_2+1}^{3+|b|+|\alpha|+N_2+M_2}\|f^{\ep^\prime}-
f^\ep\|_{\Xi,m,|b|+N_1+N_2+1,|\alpha|+M_1+M_2+1}.
$$
We conclude that the first term in (\ref{g-est}) satisfies the
hypothesis for $F$ in Lemma \ref{M-est}.
\begin{itemize}
 \item Now let us consider the second term in (\ref{g-est}):
$$
\left(\partial_{x_j}f^\ep\right)(X-Z+s(Z-Y))\big(\partial_{\xi_j}
f^\ep\big)(X-Z)f_\ep^{-2}(X-Z)=
$$
\end{itemize}
$$
=\left(\partial_{x_j}f^\ep\right)(X-Z+s(Z-Y))\left[\big(\partial_{\xi_j}
f^\ep\big)(X-Z)f_\ep^{-1}(X-Z)\right]f_\ep^{-1}(X-Z).
$$
It verifies the estimation
$$
\left|\left(\partial_{x_j}f^\ep\right)(X-Z+s(Z-Y))\big(\partial_{\xi_j}
(f^\ep)^{-1}\big)(X-Z)\right|\leq
$$
$$
\leq C(f)
\left|(f^\ep)^{-1}(X-Z)\right|\,\|f^\ep\|_{\Xi,m,1,0}\,\|f^\ep\|_{\Xi,m,0,1}\frac{\left<\xi\right>^m
\left<\eta\right>^m\left<\zeta\right>^m}{\left<\xi\right>/\left<\zeta\right>}\leq
$$
$$
\leq
C(f)\,\|f^\ep\|_{\Xi,m,1,0}\,\|f^\ep\|_{\Xi,m,0,1}\left<\eta\right>^m\left<\zeta\right>^{2m+1}\frac{\left<\xi\right>^{m-1}}
{\left<\xi\right>^m+\lambda}\leq
$$
$$
\leq
C(\nu)\,C(f)\,\|f^\ep\|_{\Xi,m,1,0}\,\|f^\ep\|_{\Xi,m,0,1}\left<\eta\right>^m
\left<\zeta\right>^{2m+1}\left<\xi\right>^{m\nu^{-1}-1}\lambda^{-\nu^{-1}}.
$$
Applying then the derivation operator
$\partial_x^a\partial_\xi^\alpha$ we proceed exactly as for the
first term. In fact the essential step is the difference of one
unit between the denominator and the numerator in
$\left<\xi\right>^{m-1}\left(\left<\xi\right>^m+\lambda\right)^{-1}$
that was obtained for the first term due to the factor
$\partial_{x_j}\partial_{\xi_j}f^\ep$ and for this second term
from the factor $\left[\big(\partial_{\xi_j}
f^\ep\big)(X-Z)f_\ep^{-1}(X-Z)\right]$ (a strictly positive
difference would have been enough). The $\ep$-continuity also
follows by the same procedure.
\begin{itemize}
 \item Now let us consider the third term in (\ref{g-est}):
$$
\left(\zeta_j-\eta_j\right)\left(\partial_{\xi_j}
f^\ep\right)(X-Z+s(Z-Y))(f^\ep)^{-1}(X-Z).
$$
\end{itemize}
Recalling the above observation we notice that the same type of
factor $\left<\xi\right>^{m-1}\left(
\left<\xi\right>^m+\lambda\right)^{-1}$ will now be obtained due
to the presence of the factor $\partial_{\xi_j} f^\ep$. The
presence of the factor $(\zeta_j-\eta_j)$ will only contribute to
modify the order of symbols in the given variables so that we
shall obtain the estimation
$$
\left|(\zeta_j-\eta_j)(\partial_{\xi_j}
f^\ep)(X-Z+s(Z-Y))(f^\ep)^{-1}(X-Z)\right|\leq
$$
$$
\leq
C(\nu)\,C(f)\,\|f^\ep\|_{\Xi,m,0,1}\left<\eta\right>^m\left<\zeta\right>^{2m}
\left<\xi\right>^{m\nu^{-1}-1}\lambda^{-\nu^{-1}}.
$$
Obviously all the following arguments given for the first term
still remain true for this third term in (\ref{g-est}).

We conclude that
$$
\underset{\ep\in
I}{\sup}\,\underset{x\in\mathcal{X}}{\sup}\left\|\left[\partial^b_x\partial^\alpha_\xi
F^\ep_\lambda\right](X;\cdot,\cdot)\right\|_{(m,N_1,M_1),(2m+1,N_2,M_2)}\leq
$$
$$
\leq
C(\nu)\,C(f)\,\|f^\ep\|_{\Xi,m,|b|+N_1+N_2+1,|\alpha|+M_1+M_2+1}^{2+|b|+|\alpha|+N_2+M_2}
\left<\xi\right>^{m\nu^{-1}-1}\lambda^{-\nu^{-1}}.
$$

Thus, choosing $\nu>m$, we have $\rho:=-m\nu^{-1}+1>0$ and thus we
can use Lemma \ref{M-est} to deduce that for each $\lambda >0$ the
application $\boldsymbol{g}(\ep;\lambda;X):=\sum_{1\leq j\leq
n}g_j(\ep;\lambda;X)$ defines an element
$\boldsymbol{g}(\cdot;\lambda;\cdot)\in
C\big(I;S^{-\rho}_1(\Xi)\big)$ and moreover we have for any
$\nu>m$
$$
\underset{\ep\in
I}{\sup}\left\|\boldsymbol{g}(\ep;\lambda;\cdot)\right\|_{\Xi,1,-\rho,N,M}\leq
C_\nu\lambda^{-\nu}.
$$
Thus, we conclude that for any $\alpha\in\mathbb{N}^n$ there
exists a constant $C_{\alpha,\nu}$ such that
$$
\underset{\ep\in
I}{\sup}\underset{x\in\mathcal{X}}{\sup}\underset{\xi\in\mathcal{X}^\prime}{\sup}\left|<\xi>^{(\rho-|\alpha|)}\big(\partial_\xi^\alpha\boldsymbol{g}\big)(\ep;\lambda;x,\xi)
\right|\leq C_{\alpha,\nu}\lambda^{-\nu}.
$$

Fixing $\ep\in I$ and using Lemma A.4 in \cite{MPR2} and the
argument at the end of the proof of Theorem 1.8 (\cite{MPR2}
Section 2.1), one obtains the conclusion of point 1 of the
Proposition.

A completely straightforward modification of Lemma A.4 in
\cite{MPR2} (just take $(\ep,x)\in I\times\mathcal{X}$ instead of
$q\in\mathcal{X}$ in the proof given in \cite{MPR2}) allows us to
conclude that for any $\lambda>0$ and any $\ep\in I$ we have that
(for any $\nu>m$)
$$
\big(\mathfrak{F}^{-1}\boldsymbol{g}\big)(\cdot;\lambda;\cdot)\in
L^1\Big((\mathcal{X};C\big(I;BC_u(\mathcal{X})\big)\Big)=\widetilde{\mathcal{L}}\quad\text{and}
\quad\|\big(\mathfrak{F}^{-1}\boldsymbol{g}\big)(\cdot;\lambda;\cdot)
\|_{\widetilde{\mathcal{L}}}\leq C_\nu\lambda^{-\nu}.
$$

Using (\ref{sharp-ep}) and the above result and repeating the
arguments at the end of the proof of Theorem 1.8 (\cite{MPR2}
Section 2.1), we obtain also the second point of the proposition.
\end{proof}
\begin{corollary}\label{cor-aff}
The map
$$
I\ni\ep\mapsto r^\ep_{\mathfrak{z}}\in\mathfrak{F}\big[\mathfrak
L\big]
$$
is continuous for the topology induced by the norm
$\|G\|_{\mathfrak F(\mathfrak L)}:=
\int_{\mathcal{X}}dx\,\|[\mathfrak F(G)](x)\|_\infty$.
\end{corollary}
\begin{proof}
We notice that for any $\widetilde{\phi}\in
C_c\Big(\mathcal{X};C\big(I;BC_u(\mathcal{X})\big)\Big)$ we have
$$
\underset{\ep\in I}{\sup}\int_{\mathcal{X}}dx\
\left\|\left[\widetilde{\phi}(x)\right](\ep)\right\|_{BC_u(\mathcal{X})}\leq
\int_{\mathcal{X}}dx\ \underset{\ep\in
I}{\sup}\left\|\left[\widetilde{\phi}(x)\right](\ep)\right\|_{BC_u(\mathcal{X})}
$$
and thus $\widetilde{\mathcal{L}}\subset
C\big(I;\mathcal{L}\big)$.
\end{proof}

\section{A continuous field of twisted crossed products}

In this section we prove Theorem \ref{T-M-1}, which in its turn implies our main result. Our proposal is
to use the concept of {\it continuous field of $C^*$-algebras}
\cite{Fe,T1,T2,Bl,Di,Ni,PR2,Ri}, and Theorem 2.4 of \cite{Ri}. Let us
notice that we are in the frame of Section 2 of \cite{Ri} with the
$C^*$-algebra $A=BC_u(\mathcal{X})$ (that is abelian and unital),
the locally compact space $\Omega = I$ (that is even compact in
our case) and the locally compact group
$G=\mathcal{X}\cong\mathbb{R}^n$ that is abelian and second
countable. The action of the group on the $C^*$-algebra is
explicitely given as the action $\theta$ of $\mathcal{X}$ on
$BC_u(\mathcal{X})$ by translations. Let us recall the notion of
{\it continuous field of $\theta$-cocycles on the group $\mathcal{X}$ over
the locally compact space $I$}, as introduced in \cite{Ri}:

\begin{definition}\label{c-f-cocycle}
A {\it continuous field of $\theta$-cocycles on $\mathcal{X}$ over $I$} is a function
$$
\widetilde{\omega}:I\times\mathcal{X}\times\mathcal{X}\rightarrow BC_u(\mathcal{X})
$$
such that:
\begin{enumerate}
\item for any $\ep\in I$ the map
$\widetilde{\omega}(\ep,\cdot,\cdot)$ defines a normalized
$\theta$-cocycle on $\mathcal{X}$, i.e.
$|[\widetilde{\omega}(\ep,y,z)](x)|=1$ and
$$
\widetilde{\omega}(\ep,x,y+z)\theta_x\left[\widetilde{\omega}(\ep,y,z)\right]=\widetilde{\omega}(\ep,x+y,z)
\widetilde{\omega}(\ep,x,y),
\quad\widetilde{\omega}(\ep,x,0)=\widetilde{\omega}(\ep,0,x)=1;
$$
\item for any $(y,z)\in\mathcal{X}\times\mathcal{X}$ the map
$I\ni\ep\mapsto\widetilde{\omega}(\ep,y,z)\in BC_u(\mathcal{X})$
is continuous; \item the map
$\mathcal{X}\times\mathcal{X}\ni(y,z)\mapsto\widetilde{\omega}(\cdot,y,z)\in
C\big(I;BC_u(\mathcal{X})\big)$ is Bochner measurable.
\end{enumerate}
\end{definition}

\begin{remark}\label{R-c-f-cocycle}
The map $\widetilde{\omega}$ defined in (\ref{tilde-omega}) and
(\ref{def-omega}) satisfies the conditions for a {\it continuous
field of $\theta$-cocycles on $\mathcal{X}$ over $I$}.
One checks easily that the first condition is satisfied, by an inspection of the explicit definition; the
last two conditions follow from Hypothesis \ref{hyp-B} that implies that $\tilde{\omega}$ belongs in fact to
$C\Big(\mathcal{X}\times\mathcal{X};C\big(I;BC_u(\mathcal{X})\big)\Big)$.
\end{remark}

In this framework we follow M. Rieffel \cite{Ri} and consider the
field of $C^*$-algebras $\left\{\mathfrak{B}^\ep\right\}_{\ep\in
I}$ over the compact space $I$ and $\mathfrak{B}$ as a
$C^*$-algebra of {\it cross-sections of this field}. Then
combining our framework with Theorem 2.4 and
Propositions 1.2 and 2.3 in \cite{Ri}, one obtains the following result:

\begin{proposition}
The family of maps
$\{\mathfrak{e}_\ep:\mathfrak{B}\rightarrow\mathfrak{B}^\ep\}_{\ep\in I}$ have the following properties:
\begin{enumerate}
\item each
$\mathfrak{e}_\ep:\mathfrak{B}\rightarrow\mathfrak{B}^\ep$ is
surjective; \item for any $F\in\mathfrak{B}$ we have
$\|F\|_{\mathfrak{B}}=\underset{\ep\in
I}{\sup}\|\mathfrak{e}_\ep(F)\|_{\mathfrak{B}^\ep}$; \item for any $F\in\mathfrak{B}$ the map
$I\ni\ep\mapsto\|\mathfrak{e}_\ep(F)\|_{\mathfrak{B}^\ep}\in\mathbb{R}_+$ is upper semi-continuous.
\end{enumerate}
\end{proposition}
Our Proposition \ref{prop-aff} now evidently implies the following Corollary
\begin{corollary}\label{upper-s-cont}
Under our Hypothesis \ref{hyp-h} and \ref {hyp-B}, the map
$I\ni\ep\mapsto\|r^\ep_\mathfrak{z}\|_{\mathfrak{B}^\ep}$ is upper semi-continuous.
\end{corollary}

Using now our Corollary \ref{cor-aff} and the Hilbert space
representation (\ref{reg-repr}) we shall obtain the lower
semi-continuity of our map $I\ni\ep\mapsto
\parallel r^\ep_{\mathfrak{z}}\parallel_\ep$ by a standard procedure.

\begin{lemma}\label{strong-cont}
Given a continuous function $I\ni\ep\mapsto\phi^\ep\in\mathfrak{L}$ and an element $\psi\in\mathcal{H}$, the map
$$
I\ni\ep\mapsto\phi^\ep\diamond^\ep\psi\in\mathcal{H}
$$ is continuous.
\end{lemma}
\begin{proof}
 For $\,\epsilon,\,\epsilon^{'}\in I$ we have
\begin{equation}\label{dif-r}
\|\phi^\epsilon\diamond^\epsilon \psi-\phi^{\epsilon^\prime}\diamond^{\epsilon^\prime}
\psi\|_{L^2(\X\times\X)}
\leq\|\phi^\epsilon\diamond^\epsilon \psi-\phi^{\epsilon^\prime}\diamond^{\epsilon}
\psi\|_{L^2(\mathcal{X}\times\mathcal{X})}+\|\phi^{\epsilon^\prime}\diamond^{\epsilon}
\psi-\phi^{\epsilon^\prime}\diamond^{\epsilon^\prime}
\psi\|_{L^2(\mathcal{X}\times\mathcal{X})}.
\end{equation}
To estimate the first term we use the definition of the $\diamond^\ep$ product in order to write
$$
\left[\phi^\epsilon\diamond^\epsilon
\psi-\phi^{\epsilon^\prime}\diamond^{\epsilon}
\psi\right](q,x)=\int_{\X}dy\left(\phi^\epsilon-\phi^{\epsilon^\prime}\right)
\left(q-\frac{x-y}{2};y\right)\psi(q+\frac{y}{2};x-y)\,\omega^{\epsilon}\left(q-\frac{x}{2};y,x-y\right),
$$
so
$$
\|\phi^\epsilon\diamond^\epsilon
\psi-\phi^{\epsilon^\prime}\diamond^{\epsilon}
\psi\|^2_{L^2(\mathcal{X}\times\mathcal{X})}\leq
\int_{\mathcal{X}}dq\int_{\mathcal{X}}dx\left[\int_{\mathcal{X}}dy\left|\left(\phi^\epsilon-
\phi^{\epsilon^\prime}\right)
\left(q-\frac{x-y}{2};y\right)\psi(q+\frac{y}{2};x-y)\right|\right]^2\leq
$$
$$
\leq
\int_{\mathcal{X}}dq\int_{\mathcal{X}}dx\left[\int_{\mathcal{X}}dy\left(\underset{z\in\mathcal{X}}
{\sup}\left|\left(\phi^\epsilon-\phi^{\epsilon^\prime}\right)
(z;y)\right|\right)\left|\psi(q+\frac{y}{2};x-y)\right|\right]^2\leq
$$
$$
\leq \|\phi^\epsilon-\phi^{\epsilon^\prime}\|_{1,\infty}
\int_{\mathcal{X}}dq\int_{\mathcal{X}}dx\int_{\mathcal{X}}dy\left(\underset{z\in\mathcal{X}}{\sup}
\left|\left(\phi^\epsilon-\phi^{\epsilon^\prime}\right)
(z;y)\right|\right)\left|\psi(q+\frac{y}{2};x-y)\right|^2=
$$
$$
=\|\phi^\epsilon-\phi^{\epsilon^\prime}\|_{\mathfrak
L}^2\,\|\psi\|^2_{L^2(\X\times\X)}
\underset{\ep'\rightarrow\ep}{\longrightarrow}0
$$
using Fubini and a change of variables. It remains to verify that
second term in (\ref{dif-r}) also converges to $0$; one has:
\begin{equation}
\|\phi^{\epsilon^\prime}\diamond^\epsilon
\psi-\phi^{\epsilon^\prime}\diamond^{\epsilon'}\psi\|_
{L^2(\mathcal{X}\times\mathcal{X})}^{2}\leq
\end{equation}
$$
\leq\int_{\X}dq\int_{\X}dx\left[\int_{\X}dy\
\left\vert\phi^{\ep^\prime}
\left(q-\frac{x-y}{2};y\right)\psi(q+\frac{y}{2};x-y)\right\vert\right.\times\hfill
$$
$$
\hfill\times\left.\left\vert\,\omega^{\epsilon}\left(q-\frac{x}{2};y,x-y\right)-
\omega^{\epsilon^{'}}\left(q-\frac{x}{2};y,x-y\right)\right\vert\right]^2
\leq
$$
$$
\leq
4\int_{\X}dq\int_{X}dx\left[\int_{\X}dy\,\|\phi^{\epsilon^\prime}(\cdot;y)\|_{\infty}\,|\psi(q+\frac{y}{2};x-y)|\right]^2\leq
$$
$$
\leq 4\,\|\phi^{\epsilon^\prime}
\|_{1,\infty}\int_{\X}dq\int_{\X}dx\int_{\X}dy\,\|\phi^{\epsilon^\prime}
(\cdot;y)\|_{\infty}\,|\psi(q+\frac{y}{2};x-y)|^2
=4\,\|\phi^{\epsilon^\prime}\|_{\mathfrak
L}^2\,\|\psi\|_{L^2(\X\times\X)}^2,
$$
by the same procedure as above. Moreover, due to our Hypothesis \ref{hyp-B} the difference
$$
\omega^{\epsilon}\left(q-\frac{x}{2};y,x-y\right)-
\omega^{\epsilon^{'}}\left(q-\frac{x}{2};y,x-y\right)
$$
converges point-wise to $0$ for $|\ep-\ep^\prime|\rightarrow0$. We conclude by the Dominated Convergence Theorem.
\end{proof}

\begin{corollary}\label{lower-s-cont}
Under our Hypothesis \ref{hyp-h} and \ref {hyp-B}, the map
$I\ni\ep\mapsto\|r^\ep_\mathfrak{z}\|_{\mathfrak{B}^\ep}$ is lower semi-continuous.
\end{corollary}
\begin{proof}
We use Corollary \ref{cor-aff}, the above Lemma \ref{strong-cont} and the well known fact that if a family
$\{S^\ep\}_{\ep\in I}$ of bounded linear operators in a Hilbert
space $\mathcal H$ is {\it strongly} continuous, then
$\ep\mapsto\parallel S^\ep\parallel_{\mathbb B(\mathcal H)}$ is
lower semi-continuous (as the supremum of a family of continuous functions).
\end{proof}

Combining the Corollaries \ref{upper-s-cont} and \ref{lower-s-cont} one proves Theorem \ref{T-M-1}
and finishes the proof of Theorem \ref{T-M}.

{\bf Acknowledgements:} M. M\u antoiu is partially supported by
{\it N\'ucleo Cientifico ICM P07-027-F "Mathematical Theory of
Quantum and Classical Magnetic Systems"} and by Chilean Science
Foundation {\it Fondecyt} under the Grant 1085162. R. Purice has
been supported by CNCSIS under the Ideas Programme, PCCE project
no. 55/2008 {\it Sisteme diferentiale in analiza neliniara si
aplicatii}. N. Athmouni thanks IMAR Bucharest for its hospitality
and Mondher Damak for his continuous support. R. Purice thanks the
Universities of Chile and Sfax for their hospitality.


\end{document}